\documentclass[12pt]{article}
\usepackage{amssymb}
\usepackage{amsfonts}
\usepackage{amsthm}
\usepackage{amsmath}
\usepackage{fullpage}
\usepackage{algorithm}
\usepackage{algorithmic}
\usepackage{mathtools}
\usepackage{enumitem}

\newtheorem{Lemma}{Lemma}

\newtheorem{Theorem}{Theorem}

\SetEnumitemKey{ncases}{itemindent=!,before=}

\newenvironment{subproof}
  {%
   \proof}
  {\endproof}

\usepackage[pdftex]{graphicx}
\graphicspath{{C:/Users/Sam/Desktop/PaperFigures/}}
\DeclareGraphicsExtensions{.pdf,.png}

\begin{document}
\title{On Arrangements of Six, Seven, and Eight Spheres: Maximal Bonding of Monatomic Ionic Compounds}
\author{Samuel Reid\thanks{Geometric Energy Corporation. $\mathsf{sam@geometricenergy.ca}$}}
\maketitle


\begin{abstract}
Let $C(n)$ be the solution to the contact number problem: the maximum number of touching pairs among any packing of 
$n$ congruent spheres in $\mathbb{R}^3$.
We prove the long conjectured values of $C(6)=12, C(7)=15$, and $C(8)=18$.
The proof strategy generalizes under an extensive case analysis to $C(9)=21, C(10) = 25, C(11) = 29, C(12) = 33$, and $C(13) = 36$.
These results have great import for condensed matter physics, materials science,
physical chemistry of interfaces, and organic crystal engineering.
\end{abstract}

\section*{The Chemical Interpretation of Contact Numbers}
The contact number problem has been extensively studied by mathematicians, condensed matter physicists, and materials
scientists \cite{BezdekReid}, \cite{ReidPacking}, \cite{ReidCrystal}
\cite{Bezdek}, \cite{ArkusManharanBrenner}, \cite{HolmesCerfon}, \cite{BezdekSep}, \cite{ReidSep}, \cite{Khan},
\cite{Szalkai}, with important applications in physics, chemistry, and
biology, \cite{lopez2003materials}, \cite{wu2010modeling}, \cite{gasser2005absorptive}, \cite{manoharan2003dense}, 
\cite{koval2014comparison}, \cite{torquato2013random}, \cite{sticky}, \cite{refa}, \cite{refb}, \cite{refc}, \cite{refd}.
This work has discovered proofs of the putative values of $C(n)$ for $n<14$.
In this paper we present explicit written proofs for $C(6)=12$, $C(7)=15$, and $C(8)=18$. 
The proofs for $8 < n < 14$ leading to the values of $C(9)=21, C(10)=25, C(11)=29, C(12)=33, C(13)=36$ are very lengthy 
and will be transcribed from a \textit{Mathematica} notebook into explicit written proofs shortly; however the 
most elegant and insightful instantiations of the proof technique are found for $5 < n < 9$, and the writing ends here 
in order to not obfuscate the main idea with hundreds of pages of case analysis.

\begin{Theorem}
Let $A_{Z}$ be an ionic monatomic $A$ compound with $Z$ atoms. Then $A_{Z}$ has at most $C(Z)$ chemical bonds. In particular,
$A_{6}$ has at 12 chemical bonds, $A_{7}$ has at most 15 chemical bonds, $A_{8}$ has at most 18 chemical bonds,
$A_{9}$ has at most 21 chemical bonds, $A_{10}$ has at most 25 chemical bonds, $A_{11}$ has at most 29 chemical bonds,
$A_{12}$ has at most 33 chemical bonds, and $A_{13}$ has at most 36 chemical bonds.
\end{Theorem}

This is true for any atom $A$, and more generally, for $Z$ non-overlapping congruent moieties;
this generalization to moieties is of great importance for organic chemistry. For functional groups
can be considered as a moiety that is approximable to a sphere, implying that steric hindrance
and the theoretical calculation of Ramachandran plots can be achieved using contact numbers
as opposed to x-ray crystallography; applied discrete geometry at work \cite{ReidCrystal}.

\section*{Six Spheres}
\begin{Theorem}
$$C(6)=12.$$
\end{Theorem}
\begin{proof}
Assume to the contrary that $C(6) \geq 13$. Then there exists a sphere packing
$$\mathcal{P} = \bigcup_{i=1}^{6} (x_{i} + \mathbb{S}^{2}) \hookrightarrow \mathbb{R}^3$$
with
$V(\mathcal{P}) = \{x_{1},x_{2},x_{3},x_{4},x_{5},x_{6}\}$ and $|E(\mathcal{P})| \geq 13$. By the Handshaking Lemma,
$$\frac{1}{2}\sum_{i=1}^{6} \deg x_{i} \geq 13, \text{ and hence, } \sum_{i=1}^{6} \deg x_{i} \geq 26.$$
Assume that $$\max_{1 \leq i \leq 6} \deg x_{i} \leq 4.$$
Hence, we obtain the following contradiction: $$\sum_{i=1}^{6} \deg x_{i} \leq 6 \cdot 4 = 24 < 26.$$
Thus, $$5 \leq \max_{1 \leq i \leq 6} \deg x_{i} < 6,$$
since $|V(\mathcal{P})| = 6$, and $\exists 1 \leq j \leq 6, \deg x_{j} = 5$.
Assume that $$\max_{i \neq j} \deg x_{i} \leq 4.$$
Hence, we obtain the following contradiction: $$\sum_{i \neq j} \deg x_{i} \leq 5 \cdot 4 = 20 < 21.$$
Thus, $$5 \leq \max_{i \neq j} \deg x_{i} < 6,$$
so there are at least two spheres of exactly degree 5. Without loss of generality, say that $\deg x_{5} = \deg x_{6} = 5$. Hence,
$$\sum_{i=1}^{4} \deg x_{i} \geq 26 - 2 \cdot 5 = 16.$$
No other spheres have degree 5, so at most 2 spheres have degree 4, and the other two spheres have at most degree 3, so
$2 \cdot 4 + 2 \cdot 3 = 14 < 16$, which is a contradiction.
Therefore, $C(6) = 12$.
\end{proof}

\section*{Seven Spheres}
\begin{Theorem}
$$C(7)=15.$$
\end{Theorem}

\begin{proof}
Assume to the contrary that $C(7) \geq 16$. Then there exists a sphere packing
$$\mathcal{P} = \bigcup_{i=1}^{7}(x_{i} + \mathbb{S}^2) \hookrightarrow \mathbb{R}^3$$
with $V(\mathcal{P}) = \{x_{1},x_{2},x_{3},x_{4},x_{5},x_{6},x_{7}\}$ and $|E(\mathcal{P})| \geq 16$. By the Handshaking
Lemma,
$$\frac{1}{2}\sum_{i=1}^{6} \deg x_{i} \geq 16, \text{ and hence, } \sum_{i=1}^{6} \deg x_{i} \geq 32.$$
Now,
$$5 \leq \max_{1 \leq i \leq 7} \deg x_{i} \leq 6.$$

\begin{Lemma}
There are at most two spheres of degree 6 in $\mathcal{P}$, 
and if there are two spheres of degree 6 in $\mathcal{P}$, 
then there can be no sphere of degree 5 in $\mathcal{P}$.
\end{Lemma}

\begin{Lemma}
If a sphere of degree 6 in $\mathcal{P}$ 
touches two spheres of degree 5 in $\mathcal{P}$,
then there is a sphere of degree 3 in $\mathcal{P}$.
\end{Lemma}

\begin{Lemma}
If there are two spheres of degree 6 in $\mathcal{P}$,
then there are at most three spheres of degree 4 in $\mathcal{P}$.
\end{Lemma}

\begin{Lemma}
If a sphere of degree 6 in $\mathcal{P}$
touches a sphere of degree 5 in $\mathcal{P}$,
then there is at most one other sphere of degree 5 in $\mathcal{P}$.
\end{Lemma}

\begin{Lemma}
Any two spheres of degree 5 in $\mathcal{P}$ which do not touch, simultaneously touch three spheres of degree 5
in $\mathcal{P}$, and every sphere of degree 5 in $\mathcal{P}$ touches four other spheres of degree 5 in $\mathcal{P}$.
\end{Lemma}

\begin{Lemma}
Any two spheres of degree 4 in $\mathcal{P}$ cannot simultaneously touch any two touching spheres of degree 5 in $\mathcal{P}$.
\end{Lemma}

\begin{enumerate}[ncases]
 \item $\displaystyle \max_{1 \leq i \leq 7} \deg x_{i} = 6$.
 \begin{subproof}
  Then, $\displaystyle \sum_{i=1}^{6} \deg x_{i} \geq 26$. Furthermore, without loss of generality,
 \begin{enumerate}
  \item $\displaystyle \max_{1 \leq i \leq 6} \deg x_{i} = 6$.
  \begin{subproof}
  Then, $\displaystyle \sum_{i=1}^{5} \deg x_{i} \geq 20$. Furthermore, without loss of generality,
  \begin{enumerate}
  \item $\displaystyle \max_{1 \leq i \leq 5} \deg x_{i} = 6$. Then, $\displaystyle \sum_{i=1}^{4} \deg x_{i} \geq 14$.
  Hence, the degree sequence is either $(6,6,6,5,3,3,3)$ or $(6,6,6,4,4,3,3)$, which contradicts Lemmas 1-4.
  \item $\displaystyle \max_{1 \leq i \leq 5} \deg x_{i} = 5$. Then, $\displaystyle \sum_{i=1}^{4} \deg x_{i} \geq 15$.
  Hence, the degree sequence is either $(6,6,5,5,4,3,3)$ or $(6,6,5,4,4,4,3)$, which contradicts Lemmas 1-4.
  \item $\displaystyle \max_{1 \leq i \leq 5} \deg x_{i} = 4$. Hence, the degree sequence is $(6,6,4,4,4,4,4)$, which
  contradicts Lemmas 1-4.
  \end{enumerate}
  \end{subproof}
  \item $\displaystyle \max_{1 \leq i \leq 6} \deg x_{i} = 5$.
  \begin{subproof}
  Then, $\displaystyle \sum_{i=1}^{5} \deg x_{i} \geq 21$. Furthermore, without loss of generality,
 $$\max_{1 \leq i \leq 5} \deg x_{i} = 5 \Rightarrow \sum_{i=1}^{4} \deg x_{i} \geq 16$$
  \begin{enumerate}
  \item $\displaystyle \max_{1 \leq i \leq 4} \deg x_{i} = 5$.
  \begin{subproof}
  Then, $\displaystyle \sum_{i=1}^{3} \deg x_{i} \geq 11$. Hence, the degree sequence is either \\
  $(6,5,5,5,5,3,3)$ or $(6,5,5,5,4,4,3)$, which contradicts Lemmas 1-4.
  \end{subproof}
  \item $\displaystyle \max_{1 \leq i \leq 4} \deg x_{i} = 4$.
  \begin{subproof}
  Then, $\displaystyle \sum_{i=1}^{3} \deg x_{i} \geq 12$. Hence, the degree sequence is \\
  $(6,5,5,4,4,4,4)$, which contradicts Lemmas 1-4.
  \end{subproof}
  \end{enumerate}
  \end{subproof}
 \end{enumerate}
 \end{subproof}
 \item $\displaystyle \max_{1 \leq i \leq 7} \deg x_{i} = 5$.
\begin{subproof}
 Then, $\displaystyle \sum_{i=1}^{6} \deg x_{i} \geq 27$. Furthermore, without loss of generality,
 \begin{align*}
 \max_{1 \leq i \leq 6} \deg x_{i} = 5 &\Rightarrow \sum_{i=1}^{5} \deg x_{i} \geq 22 \\
 \max_{1 \leq i \leq 5} \deg x_{i} = 5 &\Rightarrow \sum_{i=1}^{4} \deg x_{i} \geq 17 \\
 \max_{1 \leq i \leq 4} \deg x_{i} = 5 &\Rightarrow \sum_{i=1}^{3} \deg x_{i} \geq 12
 \end{align*}
 \begin{enumerate}
 \item $\displaystyle \max_{1 \leq i \leq 3} \deg x_{i} = 5$.
 \begin{subproof}
 Then, $\deg x_{1} + \deg x_{2} \geq 7$. Hence, the degree sequence is either \\
 $(5,5,5,5,5,5,2)$ or $(5,5,5,5,5,4,3)$, which contradicts Lemma 5.
 \end{subproof}
 \item $\displaystyle \max_{1 \leq i \leq 3} \deg x_{i} = 4$.
 \begin{subproof}
 Then $\deg x_{1} + \deg x_{2} \geq 8$, so $\deg x_{1}=\deg x_{2} = 4$.
 Hence, the degree sequence is $(5,5,5,5,4,4,4)$, which contradicts Lemma 6.
 \end{subproof}
 \end{enumerate}
\end{subproof}
\end{enumerate} 

\end{proof}

\section*{Eight Spheres}
\begin{Theorem}
$$C(8)=18.$$
\end{Theorem}
\begin{proof}
Assume to that contrary that $C(8) \geq 19$. Then there exists a sphere packing
$$\mathcal{P} = \bigcup_{i=1}^{8} (x_{i} + \mathbb{S}^2) \hookrightarrow \mathbb{R}^3$$
with $V(\mathcal{P}) = \{x_{1},x_{2},x_{3},x_{4},x_{5},x_{6},x_{7},x_{8}\}$ and $|E(\mathcal{P})| \geq 19$. By the
Handshaking Lemma,
$$\frac{1}{2}\sum_{i=1}^{8} \deg x_{i} \geq 19, \text{ and hence, } \sum_{i=1}^{8} \deg x_{i} \geq 38.$$
Now,
$$5 \leq \max_{1 \leq i \leq 8} \deg x_{i} \leq 7.$$
\begin{enumerate}[ncases]
\item $\displaystyle \max_{1 \leq i \leq 8} \deg x_{i} = 7$.
\begin{subproof}
 Then $\displaystyle \sum_{i=1}^{7} \deg x_{i} \geq 31$. Furthermore, without loss of generality,
 \begin{enumerate}
 
 \item $\displaystyle \max_{1 \leq i \leq 7} \deg x_{i} = 7$.
 \begin{subproof}
 Then $\displaystyle \sum_{i=1}^{6} \deg x_{i} \geq 24$. Furthermore, without loss of generality,
 \begin{enumerate}
 \item $\displaystyle \max_{1 \leq i \leq 6} \deg x_{i} = 7$.
 \begin{subproof}
 Then $\displaystyle \sum_{i=1}^{5} \deg x_{i} \geq 17$. Hence, the degree sequence is either \\
 $(7,7,7,5,3,3,3,3)$ or $(7,7,7,4,4,3,3,3)$, which is a contradiction.
 \end{subproof}
 \item $\displaystyle \max_{1 \leq i \leq 6} \deg x_{i} = 6$.
 \begin{subproof}
 Then $\displaystyle \sum_{i=1}^{5} \deg x_{i} \geq 18$. Hence, the degree sequence is either \\
 $(7,7,6,6,3,3,3,3)$, $(7,7,6,5,4,3,3,3)$, or $(7,7,6,4,4,4,3,3)$, which is a contradiction.
 \end{subproof}
 
 \item $\displaystyle \max_{1 \leq i \leq 6} \deg x_{i} = 5$.
 \begin{subproof}
 Then $\displaystyle \sum_{i=1}^{5} \deg x_{i} \geq 19$. Hence, the degree sequence is either \\
 $(7,7,5,5,5,3,3,3)$, $(7,7,5,5,4,4,3,3)$, or $(7,7,5,4,4,4,4,3)$, which is a contradiction
 \end{subproof}
 \end{enumerate}
 \end{subproof}
 
 \item $\displaystyle \max_{1 \leq i \leq 7} \deg x_{i} = 6$.
 \begin{subproof}
 Then $\displaystyle \sum_{i=1}^{6} \deg x_{i} \geq 25$. Furthermore, without loss of generality,
 \begin{enumerate}
 \item $\displaystyle \max_{1 \leq i \leq 6} \deg x_{i} = 6$.
 \begin{subproof}
 Then $\displaystyle \sum_{i=1}^{5} \deg x_{i} \geq 19$. Hence, the degree sequence is either \\
 $(7,6,6,6,4,3,3,3)$, $(7,6,6,5,5,3,3,3)$, or $(7,6,6,4,4,4,4,3)$, which is a contradiction.
 \end{subproof}
 \item $\displaystyle \max_{1 \leq i \leq 6} \deg x_{i} = 5$.
 \begin{subproof}
 Then $\displaystyle \sum_{i=1}^{5} \deg x_{i} \geq 20$. Hence, the degree sequence is \\
 $(7,6,5,4,4,4,4,4)$, which is a contradiction.
 \end{subproof}
 \end{enumerate}
 \end{subproof}
 
 \item $\displaystyle \max_{1 \leq i \leq 7} \deg x_{i} = 5$.
 \begin{subproof}
 Then $\displaystyle \sum_{i=1}^{6} \deg x_{i} \geq 26$. Hence, the degree sequence is either \\
 $(7,5,5,5,5,5,3,3)$, $(7,5,5,5,5,4,4,3)$, or $(7,5,5,5,4,4,4,4)$, which is a contradiction.
 \end{subproof}
 \end{enumerate}
\end{subproof}

\item $\displaystyle \max_{1 \leq i \leq 8} \deg x_{i} = 6$.
\begin{subproof}
 Then $\displaystyle \sum_{i=1}^{7} \deg x_{i} \geq 32$. Furthermore, without loss of generality,
 \begin{enumerate}
 \item $\displaystyle \max_{1 \leq i \leq 7} \deg x_{i} = 6$.
 \begin{subproof}
 Then $\displaystyle \sum_{i=1}^{6} \deg x_{i} \geq 26$. Hence, the degree sequence is either \\
 $(6,6,6,6,5,3,3,3)$, $(6,6,6,6,4,4,3,3)$, $(6,6,6,5,5,4,3,3)$, $(6,6,6,5,4,4,4,3)$, \\
 $(6,6,6,4,4,4,4,4)$, or $(6,6,5,5,4,4,4,4)$, which is a contradiction.
 \end{subproof}
 \item $\displaystyle \max_{1 \leq i \leq 7} \deg x_{i} = 5$.
 \begin{subproof}
 Then $\displaystyle \sum_{i=1}^{6} \deg x_{i} \geq 27$. Hence, the degree sequence is either \\
 $(6,5,5,5,5,4,4,4)$ or $(6,5,5,5,5,5,4,3)$, which is a contradiction.
 \end{subproof}
 \end{enumerate}
\end{subproof}

\item $\displaystyle \max_{1 \leq i \leq 8} \deg x_{i} = 5$.
\begin{subproof}
 Then $\displaystyle \sum_{i=1}^{7} \deg x_{i} \geq 33$. Hence, the degree sequence is either $(5,5,5,5,5,5,5,3)$
 or $(5,5,5,5,5,5,4,4)$, which is a contradiction.
\end{subproof}

\end{enumerate}
\end{proof}

\end{document}